\documentclass[12pt]{article}
\usepackage[T1]{fontenc}
\usepackage[latin9]{inputenc}
\usepackage[letterpaper]{geometry}
\usepackage{color}
\usepackage{amsmath}
\usepackage{amsthm}
\usepackage{amssymb}
\usepackage[unicode=true,
 bookmarks=false,
 breaklinks=false,pdfborder={0 0 1},backref=false,colorlinks=true]
 {hyperref}
\hypersetup{
 urlcolor=blue,linkcolor=blue,citecolor=blue}

\makeatletter
\numberwithin{figure}{section}
\theoremstyle{plain}
\newtheorem{thm}{\protect\theoremname}[section]
  \theoremstyle{plain}
  \newtheorem{prop}[thm]{\protect\propositionname}
  \theoremstyle{remark}
  \newtheorem{rem}[thm]{\protect\remarkname}
  \theoremstyle{plain}
  \newtheorem{lem}[thm]{\protect\lemmaname}

\usepackage{latexsym}\usepackage{mathrsfs}\usepackage{csquotes}\usepackage{enumerate}
\AtBeginDocument{%
\usepackage{pgfplots}

\renewcommand\equationautorefname{\@gobble}
}
\usepackage{color}
\usepackage{verbatim}

\numberwithin{equation}{section}

\definecolor{Red}{rgb}{1,0,0}

\definecolor{Blue}{rgb}{0,0,1}

\title{Power Lower Bounds for the Central Moments of the Last Passage Time for Directed Percolation in a Thin Rectangle}

\author{Christian Houdr\'{e}\thanks{School of Mathematics, Georgia Institute of Technology, Atlanta, Georgia 30332-0160. {\tt Email:\,houdre@math.gatech.edu}. Research supported in part by the grants \# 246283 and \# 524678 from the Simons Foundation.} \and Chen Xu \thanks{School of Mathematics, Georgia Institute of Technology, Atlanta, Georgia 30332-0160. {\tt Email:\,cxu60@math.gatech.edu}.}}
\allowdisplaybreaks

\makeatother

  \providecommand{\lemmaname}{Lemma}
  \providecommand{\propositionname}{Proposition}
  \providecommand{\remarkname}{Remark}
\providecommand{\theoremname}{Theorem}

\begin{document}
\maketitle

\makeatletter \def\blfootnote{\gdef\@thefnmark{}\@footnotetext} \makeatother


\begin{abstract}
In directed last passage site percolation with i.i.d.~random weights
with finite support over a $n\times\lfloor n^{\alpha}\rfloor$ grid,
we prove that for $n$ large enough, the order of the $r$-th central
moment, $1\le r<+\infty$, of the last passage time is lower bounded
by $n^{r(1-\alpha)/2}$, $0<\alpha<1/3$.
\end{abstract}

\section{Introduction and statements of results}

Longitudinal/shape fluctuations, i.e., the standard deviation of first/last
passage time, has attracted a lot of attention in the study of percolation
systems. It is conjectured that, on a two dimensional $n\times n$
grid, the fluctuation should be of order $n^{1/3}$ in undirected/directed
first/last passage percolation, with various weight distributions
satisfying moment conditions. However, this result has only been proved
under exponential or geometric weights, e.g., see \cite{johansson2000shape,johansson2000transversal,baik2001optimal}.
For general weight distributions satisfying moments conditions, to
date, only a upper bound of sublinear order $O(\sqrt{n/\ln n})$ (see
\cite{kesten1986aspects,kesten1993speed,benaim2008exponential,benjamini2011first,damron2014subdiffusive})
and a lower bound of order $o(\sqrt{\ln n})$ (see \cite{pemantle1994planar,newman1995divergence,zhang2008shape,auffinger2013differentiability})
have been proved in first passage percolation in various dimensions.
More is known for the directed last passage time ($DLPP$) in a thin
rectangular lattice where, via a coupling to Brownian directed percolation,
it has been shown, in \cite{bodineau2005universality}, with proper
renormalization, to converge to the Tracy-Widom distribution. Recently,
a general method to prove lower bounds for variances is devised in
\cite{chatterjee2017general}. It is applicable to first passage percolation
with continuous weights, providing a lower bound of order $o(\sqrt{\ln n})$
for the fluctuation. For a list of other results on these topics,
we refer the interested reader to the recent comprehensive survey
\cite{auffinger201550}.

In a related topic, i.e., the study of the length of the longest common
subsequences ($LCSs$) in random words, this fluctuation has also
been longed for. It is well known that $LCSs$ can be viewed as a
directed last passage percolation problem with random but dependent
weights. In \cite{lember2009standard}, the variance of the length
of $LCSs$ is shown to be linear when the letters are drawn from a
highly concentrated Bernoulli distribution. This method is further
developed in \cite{HoudreLCSVARLB2012} to show that the $r$-th moment
of $LCSs$ is of order $\Theta(n^{r/2})$ under a similarly concentrated
distribution over some finite dictionary. This power lower bound on
the fluctuation is essential in proving a Gaussian limiting law for
the length of $LCSs$. (See \cite{HoudreCLT2014})

The present paper aims at studying the $r$-th, $1\le r<+\infty$,
central moments of $DLPP$ in a thin rectangular $n\times\lfloor n^{\alpha}\rfloor$
grid. These are shown to be lower-bounded by $n^{r(1-\alpha)/2}$,
for $0<\alpha<1/3$, when $n$ is large enough. (For $r=1$, results
on the first order central moments are very sparse in the percolation
literature.) Moreover, our methodology is also applicable to first
passage time for directed site/edge percolation.

Hereafter, for convenience, $n^{\alpha}$ will be short for $\lfloor n^{\alpha}\rfloor$.
Next, the model under study is specified as follows: we consider a
$n\times n^{\alpha}$ grid having $n^{1+\alpha}$ vertices, each of
which is associated with i.i.d.~random weights $w$. We require the
weight distribution to be non-degenerate and to have finite non-negative
support, i.e., its $c.d.f.$~$F$ is such that $F(0-)=0$ and such
that there exists $C>0$ with $F(C)=1$. Then, in this setting, the
last passage time $L_{n}$ is the maximum of the sums over all the
weights, along all the unit-step up-right paths on the grid, from
$(1,1)$ to $(n,n^{\alpha})$. Namely, 
\[
L_{n}=\max_{v\in\Pi}\sum_{v\in\pi}w(v),
\]
where $\Pi$ is the set of all unit-step up-right paths from $(1,1)$
to $(n,n^{\alpha})$, and where any path $\pi\in\Pi$ is an ordered
set of vertices, i.e., $\pi=\{v_{1}=(1,1),v_{2},...,v_{n+n^{\alpha}-1}=(n,n^{\alpha})\}$
such that $v_{i+1}-v_{i}$, $i\in[n_{1}+n_{2}-1]=\{1,2,...,n_{1}+n_{2}-1\}$,
is either $\textbf{e}_{1}:=(1,0)$ or $\textbf{e}_{2}:=(0,1)$ and
where $w:\ v\rightarrow w(v)\in\mathbb{R}$ is the random weight associated
with the vertex $v\in[n]\times[n^{\alpha}]$, where $[n]:=\{1,2,...,n\}$.
Hereafter, \textit{directed path} is short for such type of path.
Further, any directed path realizing the last passage time is called
a \textit{geodesic}. Within this framework, our main result is:
\begin{thm}
\label{thm:main}The $r$-th central moment of the directed last passage
time in site percolation over a $n\times n^{\alpha}$, $0<\alpha<1/3$,
grid is lower-bounded, of order $n^{r(1-\alpha)/2}$, i.e., for $1\le r<+\infty$,
\begin{align*}
\mathbb{M}_{r}\left(L_{n}\right) & =\mathbb{E}\left(\left\lvert L_{n}-\mathbb{E}L_{n}\right\rvert ^{r}\right)\ge c_{0}n^{\frac{r(1-\alpha)}{2}},
\end{align*}
where $c_{0}>0$ is a constant which depends on $r$ but is independent
of $n$.
\end{thm}

The remaining of this paper is dedicated to the proof of the above
theorem and is organized as follows: at the beginning of the next
section, we show that with high probability the number of $hi$-mode
weights (to be defined) on any geodesic grows at most linearly in
$n$. More importantly, this indicates that there exist at least linearly
many $lo$-mode weights on any geodesic. In turn, this helps showing
that if $L_{n}$ is represented as a random function of the number
of $lo$-mode weights over the grid, then with high probability it
locally satisfies a reversed Lipschitz condition. In Section \ref{Sec:mainproof},
the proof of the main theorem is completed by showing how{\small{}
such a }local reversed Lipschitz condition ensures a power lower bound
for any central moment. In the concluding section, we briefly discuss
the potential extension of our proof to the case of the second order
central moment, i.e., the variance over a square grid, i.e., $\alpha=1$.

\section{Preliminaries\label{sec:Preliminaries}}

We start by introducing the notions of $hi/lo$ mode of site weights:
since the weight distribution is non-degenerate and non-negative,
there exists $m>0$ such that $\mathbb{P}(w>m)=p>0$ and $\mathbb{P}(w\le m)=1-p>0$.
Then, $w$ is said to be in $hi$ mode if $w>m$; otherwise, $w$
is in $lo$ mode. In addition, let $M_{n}$ be the maximum of the
number of weights in $hi$ mode over all directed paths:
\[
M_{n}=\max_{v\in\Pi}\sum_{v\in\pi}\boldsymbol{1}\left(w(v)>m\right),
\]
which is the same as the last passage time for the same grid with
Bernoulli weights $\boldsymbol{1}(w(v)>m)$. In this section, on an
explicitly constructed event of very high probability, $L_{n}$ is
shown to locally satisfy a reversed Lipschitz condition, where now
$L_{n}$ is considered as a function of the number of $hi$ mode weights
over the grid.

\subsection{Linear Growth of $M_{n}$}

First, we show that there exists an absolute constant $0<c_{1}<1$
such that the probability that $M_{n}$ is larger than $c_{1}n$ is
exponentially small.
\begin{prop}
\label{prop:concentcn}There exist constants $0<c_{1}<1$ and $0<c_{2}<+\infty$,
independent of $n$, such that 

\[
\mathbb{P}(M_{n}\ge c_{1}n)\le\exp(-c_{2}n),
\]
for $n$ large enough.
\end{prop}

To prove Proposition \ref{prop:concentcn}, we start by showing a
concentration inequality for $M_{n}$. The proof, via the entropy
method is akin to the proof of Theorem $3.12$ described in \cite{auffinger201550}.
\begin{prop}
\label{prop:concent}There exists $0<c_{3},c_{4}<\infty$ such that
for $t\in(0,c_{4}\sqrt{n+n^{\alpha}-1})$,

\[
\mathbb{P}(M_{n}-\mathbb{E}M_{n}\ge t\sqrt{n+n^{\alpha}-1})\le\exp(-c_{3}t^{2}).
\]
\end{prop}

\begin{proof}
Let $\psi(\lambda)=\log\mathbb{E}\exp(\lambda(M_{n}-\mathbb{E}M_{n}))$.
Then, as shown next, it suffices to show that for some $c>0$ and
$\lambda\in(0,c)$,
\begin{equation}
\psi(\lambda)\le c(n+n^{\alpha}-1)\lambda^{2}.\label{eq:philambda}
\end{equation}
Indeed, for any $\lambda>0$,
\begin{eqnarray*}
\mathbb{P}\left(M_{n}-\mathbb{E}M_{n}\ge\sqrt{n+n^{\alpha}-1}t\right) & \mathbb{\le} & \mathbb{P}(\exp(\lambda(M_{n}-\mathbb{E}M_{n}))\ge\exp(t\lambda\sqrt{n+n^{\alpha}-1}))\\
 & \mathbb{\le} & \exp\left(\psi(\lambda)-t\lambda\sqrt{n+n^{\alpha}-1}\right)\\
 & \le & \exp\left(c(n+n^{\alpha}-1)\lambda^{2}-t\lambda\sqrt{n+n^{\alpha}-1}\right).
\end{eqnarray*}
Letting $\lambda=t\sqrt{n+n^{\alpha}-1}/2c$ will complete the proof,
wherever (\ref{eq:philambda}), which we proceed to prove next, holds
true. For any non-negative random variable $X$ (and the convention
$0\ln0=0$), let 
\[
EntX=\mathbb{E}X\log X-\mathbb{E}X\log\mathbb{E}X.
\]
Then, 
\begin{eqnarray*}
\frac{d}{d\lambda}\left(\frac{\psi(\lambda)}{\lambda}\right) & = & \frac{d}{d\lambda}\left(\frac{1}{\lambda}\ln\mathbb{E\exp}(\lambda(M_{n}-\mathbb{E}M_{n}))\right)\\
 & = & -\frac{1}{\lambda^{2}}\ln\mathbb{E}\exp(\lambda(M_{n}-\mathbb{E}M_{n}))+\frac{1}{\lambda}\frac{\mathbb{E}(M_{n}-\mathbb{E}M_{n})\exp(\lambda(M_{n}-\mathbb{E}M_{n}))}{\mathbb{E}\exp(\lambda(M_{n}-\mathbb{E}M_{n}))}\\
 & = & -\frac{1}{\lambda^{2}}\ln\mathbb{E}\exp(\lambda M_{n})\cdot\exp(-\lambda\mathbb{E}M_{n})+\frac{\mathbb{E}(M_{n}-EM_{n})\exp(\lambda M_{n})}{\lambda\mathbb{E}\exp(\lambda M_{n})}\\
 & = & -\frac{1}{\lambda^{2}}\left(\ln\mathbb{E}\exp(\lambda M_{n})-\lambda\mathbb{E}M_{n}\right)+\frac{\mathbb{E}(M_{n}-EM_{n})\exp(\lambda M_{n})}{\lambda\mathbb{E}\exp(\lambda M_{n})}\\
 & = & \frac{\mathbb{E}M_{n}}{\lambda}-\frac{1}{\lambda^{2}}\ln\mathbb{E}\exp(\lambda M_{n})+\frac{\mathbb{E}L\exp(\lambda M_{n})}{\lambda\mathbb{E}\exp(\lambda M_{n})}-\frac{\mathbb{E}M_{n}}{\lambda}\\
 & = & \frac{\lambda\mathbb{E}M\exp(\lambda M_{n})-\mathbb{E}\exp(\lambda M_{n})\ln\mathbb{E}\exp(\lambda M_{n})}{\lambda^{2}\mathbb{E}\exp(\lambda M_{n})}\\
 & = & \frac{Ent\exp(\lambda M_{n})}{\lambda^{2}\mathbb{E}\exp(\lambda M_{n})}.
\end{eqnarray*}
If 
\begin{equation}
Ent\exp(\lambda M_{n})\le c(n+n^{\alpha}-1)\lambda^{2}\mathbb{E}\exp(\lambda M_{n}),\label{eq:entropy}
\end{equation}
for $\lambda\in(0,c)$, then we would have
\begin{align*}
\frac{d}{d\lambda}\left(\frac{\psi(\lambda)}{\lambda}\right) & =\frac{Ent\exp(\lambda M_{n})}{\lambda^{2}\mathbb{E}\exp(ML_{n})}\\
 & \le c(n+n^{\alpha}-1),
\end{align*}
for which, it would follow that $\psi(\lambda)\le c(n+n^{\alpha}-1)\lambda^{2}$.
Let us therefore prove (\ref{eq:entropy}). First, enumerate the $n^{1+\alpha}$
vertices as $v_{1},v_{2},...,v_{n^{1+\alpha}}$ and denote the associated
Bernoulli weights as $w(v_{i})$, i.e., the indicator function of
whether $v_{i}$ is in $hi$-mode. By the tensorization property of
the entropy,
\begin{equation}
Ent\exp(\lambda M)\le\sum_{i=1}^{n^{1+\alpha}}\mathbb{E}Ent_{i}\exp(\lambda M),\label{eq:entropytensor}
\end{equation}
where $Ent_{i}(\cdot)$ is the entropy taken only relative to the
random weight $w(v_{i})$. Now, recall (see \cite[Theorem 6.15]{boucheron2013concentration}):
for all $t\in\mathbb{R}$,
\[
Ent\exp(tX)\le\mathbb{E}\left(\exp(tX)q(-t(X-X')_{+})\right),
\]
where $q(x)=x(e^{x}-1)$ and where $X'$ is an independent copy of
$X$ . Therefore, (\ref{eq:entropytensor}) and (\ref{eq:tensorentropyapplied})
lead to
\begin{equation}
Ent\exp(\lambda M)\le\sum_{i=1}^{n^{1+\alpha}}\mathbb{E}\left(\exp(\lambda M)q(-\lambda(M-M_{i}')_{+})\right).\label{eq:tensorentropyapplied}
\end{equation}
However, it is clear that $M-M_{i}'\le1$ with equality if and only
if $w(v_{i})=1$ and, its independent copy, $w'(v_{i})=0$, for $v_{i}\in\mathcal{G}$,
where $\mathcal{G}$ is the set of vertices in the intersection of
all the geodesics, i.e., $\mathcal{G}=\cap_{geodesics}\{v\in geodesic\}$.
So it follows that 
\[
(M-M_{i}')_{+}\le1-w'(v_{i}),
\]
which in turn yields that 
\[
-\lambda(M-M_{i}')_{+}\ge-\lambda(1-w'(v_{i})).
\]
On the other hand, $q'(x)=xe^{x}+e^{x}-1<0$, when $x<0$, and so
\[
q(-\lambda(M-M_{i}')_{+})\le q(-\lambda(1-w'(v_{i})).
\]
Moreover, $q(0)=0$ gives us 
\[
\mathbb{E}\left(\exp(\lambda M)q(-\lambda(M-M_{i}')_{+})\right)=\mathbb{E}\left(\exp(\lambda M)q(-\lambda(M-M_{i}')_{+})\boldsymbol{1}(v_{i}\in\mathcal{G})\right).
\]
 Thus,
\begin{align*}
Ent\exp(\lambda M) & \le\sum_{i=1}^{n^{1+\alpha}}\mathbb{E}\left(\exp(\lambda M)q(-\lambda(M-M_{i}')_{+})\boldsymbol{1}(v_{i}\in\mathcal{G})\right)\\
 & \le\sum_{i=1}^{n^{1+\alpha}}\mathbb{E}\left(\exp(\lambda M)q(-\lambda(1-w'(v_{i}))\boldsymbol{1}(v_{i}\in\mathcal{G})\right)\\
 & =\sum_{i=1}^{n^{1+\alpha}}\mathbb{E}\left(\exp(\lambda M)1(v_{i}\in\mathcal{G})\right)\mathbb{E}q(-\lambda(1-w(v_{i}))\\
 & =Card(\mathcal{G})\mathbb{E}q(-\lambda(1-w(v_{1}))\mathbb{E}\left(\exp(\lambda M)\right).
\end{align*}
Since any geodesic covers exactly $n+n^{\alpha}-1$ vertices, $Card(\mathcal{G})\le n+n^{\alpha}-1$,
and
\begin{equation}
Ent\exp(\lambda M)\le(n+n^{\alpha}-1)\mathbb{E}q(-\lambda(1-w(e_{1}))\mathbb{E}\exp(\lambda M).\label{eq:entropyupbd}
\end{equation}
Now, by dominated convergence,
\begin{align}
\lim_{\lambda\searrow0}\frac{\mathbb{E}q(-\lambda(1-w(v_{1}))}{\lambda^{2}} & =\mathbb{E}\left(\lim_{\lambda\searrow0}\frac{(1-w(v)(1-\exp(-\lambda(1-w(v_{1})))}{\lambda}\right)\nonumber \\
 & =\mathbb{E}(1-w(v_{1}))^{2}=1-p.\label{eq:dominconv}
\end{align}
Hence, there exists $c$ such that when $\lambda\in(0,c)$, $\mathbb{E}q(-\lambda(1-w(v_{1}))\le\lambda^{2}$.
Combining (\ref{eq:dominconv}) with (\ref{eq:entropyupbd}), it finally
follows that
\[
Ent\exp(\lambda M)\le(n+n^{\alpha}-1)\lambda^{2}\mathbb{E}\exp(\lambda M),
\]
for $\lambda\in(0,c)$.
\end{proof}
\begin{rem}
Note that in Proposition \ref{prop:concent}, and in contrast to \cite[Theorem 1.1]{damron2014subdiffusive},
the subcritical condition, i.e., $p<p_{c}$, where $p_{c}$ is the
critical probability in directed bond percolation, in two dimensions,
is not required. This is mainly due to the fact that the subcritical
condition is needed there to bound the length of the geodesics in
undirected percolation; however, in our directed case, any directed
path is naturally of length $n+n^{\alpha}-1$.
\end{rem}

\begin{proof}[$\text{\textbf{Proof of Proposition}}$ \ref{prop:concentcn}:]
 Let $g$ be the shape function, i.e., let $g((1,a))=\lim_{n\rightarrow+\infty}\allowbreak\mathbb{E}M(n,na)/n$,
where $M(n,na)$ is the last passage time over a $n\times na$ grid.
It is shown in \cite{martin2004limiting} that $g((1,a))=p+2\sqrt{p(1-p)a}+o(\sqrt{a})$,
as $a\rightarrow0$. Hence, there exists $N$ such that for $n>N$,
$\mathbb{E}M(n,n^{\alpha})\le(p+1)n/2$, which, when combined with
Proposition \ref{prop:concent}, gives $\mathbb{P}(M\ge(p+1)n/2+t\sqrt{n+n^{\alpha}-1})\le\exp(-c_{1}t^{2})$,
for any $t\in(0,c_{4}\sqrt{n+n^{\alpha}-1})$. 

Further, let $0<\text{\ensuremath{\varepsilon}}<(1-p)/2$. Then there
exists a constant $0<\varepsilon'<c_{4}$, independent of $n$, such
that if $t=\varepsilon'\sqrt{n+n^{\alpha}-1}\in(0,c_{4}\sqrt{n+n^{\alpha}-1})$,
then $t\sqrt{n+n^{\alpha}-1}\le\varepsilon n$ and $t^{2}=(\varepsilon')^{2}(n+n^{\alpha}-1)>(\varepsilon')^{2}n$.
Hence, for this particular $t$, $\mathbb{P}(M\ge(\varepsilon+(p+1)/2)n)\le\exp(-c_{3}(\varepsilon')^{2}n)$.
Setting $c_{1}=(\varepsilon+(p+1)/2)<1$ and $c_{2}=c_{3}(\varepsilon')^{2}>0$,
finishes the proof.
\end{proof}

\subsection{Local Reversed Lipschitz Condition}

To begin with, let us set the underlying probability space as $\Omega_{n}=\mathbb{R}^{n^{1+\alpha}}$
associated with the product measure $\bigotimes_{i=1}^{n^{1+\alpha}}F$
and let $W=(w(v_{i}))_{i=1}^{n^{1+\alpha}}$ be the random vector
of weights under an arbitrary but deterministic enumeration of weights
over all the $n^{1+\alpha}$ vertices. Let $N$ be the total number
of $v_{i}$ such that $w(v_{i})$ is in $hi$ mode and so, clearly,
$N$ is a binomial variable with parameters $n^{1+\alpha}$ and $p$.
In addition, any weight $w$ can be decided in a two-step way: it
is first fixed to be in $hi/lo$ mode by flipping a Bernoulli random
variable with parameter $p$; then it is further associated with a
non-negative real value by drawing from $F$ conditioned on the fixed
$hi/lo$ mode in the first step. Based on this point of view, one
can construct an iterative scheme to decide $W$ by starting from
a grid with all the weights in $lo$ mode and changing them into $hi$
mode one by one until after some deliberate random steps.

To be more precise, a (finite) sequence of random vectors of weights
$\{W^{k}=(w^{k}(v_{i}))_{i=1}^{n^{1+\alpha}}\}_{k=0}^{n^{1+\alpha}}$
is iteratively defined as follows: First, let $W^{0}=\{w^{0}(v_{i})\}_{i=1}^{n^{1+\alpha}}$,
where $w^{0}(v_{i})$ has distribution $F$ conditioned on being in
$lo$ mode. Then, $W^{0}$ is clearly identical, in distribution,
to $W$ conditioned on $N=0$. Second, once $W^{k}$ is defined, one
vertex $v_{i_{0}}$ is uniformly chosen at random from the set $\{v_{i}:\ w^{k}(v_{i})\ in\ lo\ mode\}$
and then $W^{k+1}$ is defined such that $w^{k+1}(v_{i_{0}})$ is
sampled from $F(\cdot)$ conditioning on being in $hi$ mode and $w^{k+1}(v_{i})=w^{k}(v_{i})$
for $i\neq i_{0}$, i.e., $W^{k+1}$ is defined by changing one uniformly
chosen $lo$-mode weight in $W^{k}$ to a $hi$-mode weight. The second
step is repeated $n^{1+\alpha}$ times until all $lo$-mode weights,
in $W^{0}$, are changed to only $hi$-mode weights in $W^{n^{1+\alpha}}$.

By the very definition, for $0\le k\le n^{1+\alpha}$, there are $k$
$lo$-mode weights in $W^{k}$. Moreover, $\{W^{k}\}_{k=0}^{n^{1+\alpha}}$
are dependent but independent of both $W$ and $N$. Next, we show
that $W^{k}$ has the same law as $W$ conditioned on $N=k$.
\begin{lem}
\label{lem:eqdist}For any $k=0,1,...,n^{1+\alpha}$,
\begin{equation}
W^{k}=_{d}(W\ |\ N=k),\label{eq:condeqdis}
\end{equation}
and moreover, 
\begin{equation}
W^{N_{1}}=_{d}W,\label{eq:eqdis}
\end{equation}
where $=_{d}$ denotes equality in distribution.
\end{lem}

\begin{proof}
The proof is by induction on $k$. By definition, $W^{0}=_{d}W$ conditioned
on $N=0$. Assume now that (\ref{eq:condeqdis}) is true for $k$,
i.e., that for any $(\omega_{i})_{i=1}^{n^{1+\alpha}}\in\Omega_{n}$
such that $Card\left(\{\omega_{i}\ in\ lo\ mode\}_{i=1}^{n^{1+\alpha}}\right)=k$,
\begin{equation}
\mathbb{P}\left(W^{k}=(\omega_{i})_{i=1}^{n^{1+\alpha}}\right)=\left(\begin{array}{c}
n^{1+\alpha}\\
k
\end{array}\right)^{-1}.\label{eq:indudistr}
\end{equation}
Then, for any $(\omega_{i})_{i=1}^{n^{1+\alpha}}\in\Omega$ such that
$Card\left(\{\omega_{i}\ in\ lo\ mode\}_{i=1}^{n^{1+\alpha}}\right)=k+1$,
\begin{equation}
\mathbb{P}\left(W^{k+1}=(\omega_{i})_{i=1}^{n^{1+\alpha}}\right)=\sum_{j=1}^{k+1}\mathbb{P}\left(W^{k+1}=(\omega_{i})_{i=1}^{n^{1+\alpha}}|B_{j}^{k+1}\right)\mathbb{P}(B_{j}^{k+1}),\label{eq:indudistr2}
\end{equation}
where $B_{j}^{k+1}$, $1\le j\le k+1$, denotes the event that the
$jth$ weight $1$ in $\{\omega_{i}^{k+1}:\omega_{i}^{k+1}=1,\ 1\le i\le n^{1+\alpha}\}$
is the one which has been flipped uniformly at random from the weight
$0$ in $W^{k}$. Combining (\ref{eq:indudistr}) and (\ref{eq:indudistr2})
gives
\begin{eqnarray*}
\mathbb{P}\left(W^{k+1}=(\omega_{i})_{i=1}^{n^{1+\alpha}}\right) & = & \sum_{j=1}^{k+1}\left(\begin{array}{c}
n^{1+\alpha}\\
k
\end{array}\right)^{-1}\frac{1}{n^{1+\alpha}-k}\\
 & = & \frac{k!\left(n^{1+\alpha}-k\right)}{\left(n^{1+\alpha}\right)!}\frac{k+1}{n^{1+\alpha}-k}\\
 & = & \left(\begin{array}{c}
n^{1+\alpha}\\
k+1
\end{array}\right)^{-1}.
\end{eqnarray*}
Next, (\ref{eq:condeqdis}) and the independence of $N$ and $\{W^{k}\}_{k=0}^{n^{1+\alpha}}$
give
\begin{eqnarray*}
\mathbb{E}\left(\exp(i\langle t,W\rangle)\right) & \mathbb{=} & \sum_{k=0}^{n^{1+\alpha}}\mathbb{E}\left(\exp\left(i\langle t,W\rangle\right)|N=k\right)\mathbb{P}(N=k)\\
 & = & \sum_{k=0}^{n^{1+\alpha}}\mathbb{E}\left(\exp\left(i\langle t,W^{k}\rangle\right)|N=k\right)\mathbb{P}(N=k)\\
 & = & \sum_{k=0}^{n^{1+\alpha}}\mathbb{E}\left(\exp\left(i\langle t,W^{N}\rangle\right)|N=k\right)\mathbb{P}(N=k)\\
 & = & \mathbb{E}\left(\exp\left(i\langle t,W^{N}\rangle\right)\right).
\end{eqnarray*}
\end{proof}
This particular way of iterative sampling provides a new perspective
on $L_{n}$. Letting $L_{n}(k):=L_{n}(W^{k})$ and $L_{n}:=L_{n}(W)$
be respectively the last passage times under weights settings $W^{k}$
and $W$, it is clear from Lemma \ref{lem:eqdist}, that $L_{n}(N)=_{d}L_{n}$
and so it is equivalent to study $\mathbb{M}_{r}(L_{n}(N))$ or $\mathbb{M}_{r}(L_{n})$.
We finish this section by showing that on an event of probability
exponentially close to $1$, $\{L_{n}(k)\}_{i=1}^{n^{1+\alpha}}$
satisfies locally a reversed Lipschitz condition.
\begin{lem}
\label{lem:lbOn}There exist positive constants $c_{2}$, $c_{5}$
and $c_{6}$ not depending on $n$ such that, when $n$ is large enough,
\begin{align*}
 & \mathbb{P}\left(O_{n}:=\bigcap_{\begin{array}{c}
i,j\in I,\ j\ge i+c_{6}\sqrt{p(1-p)n^{1+\alpha}}\end{array}}\left\{ L_{n}(j)-L_{n}(i)\ge\frac{c_{5}}{n^{\alpha}}(j-i)\right\} \right)\\
\ge & \ 1-12p(1-p)n^{1+\alpha}\exp(-c_{2}n)-p(1-p)n^{1+\alpha}\exp\left(-\frac{c_{5}^{2}c_{6}\sqrt{p(1-p)}}{4}n^{\frac{1-3\alpha}{2}}\right),
\end{align*}
where $I=\left(n^{1+\alpha}p-\sqrt{(1-p)pn^{1+\alpha}},\ n^{1+\alpha}p+\sqrt{(1-p)pn^{1+\alpha}}\right)$.
\end{lem}

\begin{proof}
Define a set $B_{n}=\{w:\ w\in\Omega_{n},\ M_{n}(w)<c_{1}n\}$ and
so, by Proposition \ref{prop:concentcn}, $\mathbb{P}(B_{n})\ge1-\exp(-c_{2}n)$,
when $n$ is large enough. Further, let $A_{n}:=\{W\in B_{n}\}$ and
$A_{n}^{k}:=\{W^{k}\in B_{n}\}$. Then, by Lemma \ref{lem:eqdist},
\begin{equation}
\mathbb{P}\left(\left(\bigcap_{k\in I}A_{n}^{k}\right)^{c}\right)\le\sum_{k\in I}\mathbb{P}\left(\left(A_{n}^{k}\right)^{c}\right)=\sum_{k\in I}\mathbb{P}\left(A_{n}^{c}\ |\ N_{1}=k\right)\le\sum_{k\in I}\frac{\mathbb{P}(A_{n}^{c})}{\mathbb{P}(N_{1}=k)}.\label{eq:eventsA}
\end{equation}
Meanwhile, for any $k\in I$, $\mathbb{P}(N_{1}=k)\ge1/(6\sqrt{n^{1+\alpha}p(1-p)})$.
Indeed,
\begin{align*}
 & \mathbb{P}(N_{1}=k)\\
\ge & \min\left(\mathbb{P}\left(N_{1}=pn^{1+\alpha}-\lfloor\sqrt{n^{2+\alpha-\delta}}\rfloor\right),\ \mathbb{P}\left(N_{1}=pn^{1+\alpha}+\lfloor\sqrt{n^{2+\alpha-\delta}}\rfloor\right)\right),
\end{align*}
and, by de Moivre\textendash Laplace Theorem,
\begin{eqnarray*}
\mathbb{P}(N_{1}=pn^{1+\alpha}-\lfloor\sqrt{n^{2+\alpha-\delta}}\rfloor) & \ge & \frac{1}{2\sqrt{n^{1+\alpha}p(1-p)}}\exp\left(-\frac{\left(\lfloor\sqrt{n^{2+\alpha-\delta}}\rfloor\right)^{2}}{(1-p)pn^{1+\alpha}}\right)\\
 & \ge & \frac{1}{2\sqrt{n^{1+\alpha}p(1-p)}}\exp\left(-\frac{n^{1-\delta}}{(1-p)p}\right),
\end{eqnarray*}
when $n$ is large enough. Similarly, this lower bound also holds
for $\mathbb{P}(N_{1}=pn^{1+\alpha}+\lfloor\sqrt{n^{2+\alpha-\delta}}\rfloor)$
and therefore
\begin{equation}
\mathbb{P}(N_{1}=k)\ge\frac{1}{2\sqrt{n^{1+\alpha}p(1-p)}}\exp\left(-\frac{n^{1-\delta}}{(1-p)p}\right),\label{eq:N=00003Dkproblb}
\end{equation}
for any $k\in I$. Combining (\ref{eq:eventsA}) and (\ref{eq:N=00003Dkproblb})
gives:
\begin{eqnarray}
\mathbb{P}\left(\left(\bigcap_{k\in I}A_{n}^{k}\right)^{c}\right) & \le & 2\sqrt{n^{2+\alpha-\delta}}2\sqrt{n^{1+\alpha}p(1-p)}\mathbb{P}(A_{n}^{c})\nonumber \\
 & \le & 4p(1-p)n^{1+\alpha-\delta/2}\exp\left(n\left(-c_{2}+\frac{1}{n^{\delta}(1-p)p}\right)\right).\label{eq:Anup}
\end{eqnarray}
Next, before building a martingale difference sequence, we show that,
with high probability, the difference between $L_{n}(k+1)$ and $L_{n}(k)$
conditioned on $W^{k}$ can be lower bounded by a fractional polynomial
in $n$. Indeed, it always holds true that
\[
\mathbb{E}\left(L_{n}(k+1)-L_{n}(k)|W^{k}\right)\ge\frac{n+n^{\alpha}-M_{n}(k)}{n^{1+\alpha}-k}\left(\mathbb{E}(w|hi)-m\right),
\]
since $L_{n}(k+1)$ increases if and only if the chosen $lo$-mode
weight is on any geodesic under $W^{k}$. Note that there are at least
$(n+n^{\alpha}-M_{n}(k))$ many $lo$-mode weights on any geodesic
and $(n^{1+\alpha}-k)$ many $lo$-mode weights over the grid under
$W^{k}$, so the probability that any $lo$-mode weight on some geodesic
is chosen is at least $\left(n+n^{\alpha}-M_{n}(k)\right)/\left(n^{1+\alpha}-k\right)$.
In addition, the expected increment of a single flipping should be
$\left(\mathbb{E}(w|hi)-m\right)>0$. Hence, by conditioning on $A_{n}^{k}=\{M_{n}(k)<c_{1}n\}$,
\begin{equation}
\mathbb{E}\left(L_{n}(k+1)-L_{n}(k)|W^{k}\right)\ge\frac{(1-c_{1})}{n^{\alpha}}\left(\mathbb{E}(w|hi)-m\right).\label{eq:martingaleroot}
\end{equation}
Based on this lower bound, a martingale difference sequence is built
as follows: for each $k\ge0$, letting
\begin{eqnarray*}
\Delta_{k+1} & = & \begin{cases}
L_{n}(k+1)-L_{n}(k), & when\ A_{n}^{k}\ holds,\\
(1-c_{1})\left(\mathbb{E}(w|hi)-m\right)/n^{\alpha} & otherwise.
\end{cases}
\end{eqnarray*}
Therefore, letting $c_{5}:=(1-c_{1})\left(\mathbb{E}(w|hi)-m\right)$,
\begin{equation}
\mathbb{E}\left(\Delta_{k+1}|W^{k}\right)\ge\frac{c_{5}}{n^{\alpha}}.\label{eq:difflb}
\end{equation}
Now, for each $k=0,1,...,n^{1+\alpha}$, let $\mathcal{F}_{k}:=\sigma(W^{0},W^{1},..,W^{k})$,
be the $\sigma$-field generated by $W^{0}$, $W^{1}$,...,$W^{k}$.
Clearly, $\{\Delta_{k}-\mathbb{E}(\Delta_{k}|\mathcal{F}_{k-1}),\mathcal{\mathcal{F}}_{k}\}_{1\le k\le n^{1+\alpha}}$
forms a martingale differences sequence and since $0\le\text{\ensuremath{\Delta}}_{k}\le C$
and thus $-C\le\Delta_{k}-\mathbb{E}(\Delta_{k}\vert\mathcal{F}_{k-1})\le C$,
Hoeffding's martingale inequality gives, for any $i<j$,
\begin{align*}
\mathbb{P}\left(\sum_{k=i+1}^{j}(\Delta_{k}-\mathbb{E}(\Delta_{k}|\mathcal{F}_{k-1})<-\frac{c_{5}}{2n^{\alpha}}(j-i)\right) & \le\exp\left(-\frac{2c_{5}^{2}(j-i)^{2}}{4n^{2\alpha}\sum_{k=i+1}^{j}C^{2}}\right)\\
 & =\exp\left(-\frac{c_{5}^{2}(j-i)}{2n^{2\alpha}C^{2}}\right).
\end{align*}
Moreover, from (\ref{eq:difflb}), 
\[
\sum_{k=i+1}^{j}\mathbb{E}\left(\Delta_{k}|W^{k}\right)\ge\frac{c_{5}}{n^{\alpha}},
\]
 and therefore,
\begin{eqnarray}
\mathbb{P}\left(\sum_{k=i+1}^{j}\Delta_{k}\le\frac{c_{5}}{2n^{\alpha}}\right) & \le & \mathbb{P}\left(\sum_{k=i+1}^{j}(\Delta_{k}-\mathbb{E}(\Delta_{k}|\mathcal{F}_{k-1}))<-\frac{c_{5}}{2n^{\alpha}}(j-i)\right)\nonumber \\
 & \le & \exp\left(-\frac{c_{5}^{2}(j-i)}{2n^{2\alpha}C^{2}}\right).\label{eq:diffsumbd}
\end{eqnarray}
For each $n\ge1$, set 
\[
O_{n}^{\Delta}=\bigcap_{i,j\in I,\ j\ge i+\ell(n)}\left\{ \sum_{k=i+1}^{j}\Delta_{k}\ge\frac{c_{5}}{2n^{\alpha}}(j-i)\right\} ,
\]
where $\ell(n)\ge0$ will be fixed later. Then, by (\ref{eq:diffsumbd}),
\begin{eqnarray}
\mathbb{P}\left((O_{n}^{\Delta})^{c}\right) & \le & \sum_{i,j\in I,\ j\ge i+\ell(n)}\mathbb{P}\left(\sum_{k=i+1}^{j}\Delta_{k}<\frac{c_{5}}{2n^{\alpha}}(j-i)\right)\nonumber \\
 & \le & Card(I)^{2}\exp\left(-\frac{c_{5}^{2}\ell(n)}{2n^{2\alpha}C^{2}}\right)\nonumber \\
 & = & n^{2+\alpha-\delta}\exp\left(-\frac{c_{5}^{2}\ell(n)}{2n^{2\alpha}C^{2}}\right).\label{eq:Onup}
\end{eqnarray}
Next, by $\sum_{k=i+1}^{j}\Delta_{k}=L_{n}(j)-L_{n}(i)$ conditioned
on $A_{n}^{k}$ for $k\in[i,j]$ and the very definitions of $\Delta_{k}$
and $O_{n}$, 
\[
\left(\bigcap_{k\in I}A_{n}^{k}\right)\cap O_{n}^{\Delta}\subseteq O_{n}
\]
Therefore, combining (\ref{eq:Anup}) and (\ref{eq:Onup}) and letting
$\ell(n)=c_{6}\sqrt{n^{2+\alpha-\delta}}$ gives
\begin{align}
\mathbb{P}\left((O_{n})^{c}\right) & \le\mathbb{P}\left(\left(\bigcap_{k\in I}A_{n}^{k}\right)^{c}\right)+\mathbb{\mathbb{P}}\left((O_{n}^{\Delta})^{c}\right)\nonumber \\
 & \le4p(1-p)n^{1+\alpha-\delta/2}\exp\left(n\left(-c_{2}+\frac{1}{n^{\delta}(1-p)p}\right)\right)+n^{2+\alpha-\delta}\exp\left(-\frac{c_{5}^{2}\ell(n)}{2n^{2\alpha}C^{2}}\right)\nonumber \\
 & =4p(1-p)n^{1+\alpha-\delta/2}\exp\left(n\left(-c_{2}+\frac{1}{n^{\delta}(1-p)p}\right)\right)+n^{2+\alpha-\delta}\exp\left(-\frac{c_{5}^{2}c_{6}}{2C^{2}}n^{1-3\alpha/2-\delta}\right).\label{eq:probofOc}
\end{align}
Clearly, the right hand side of (\ref{eq:probofOc})converges, to
$0$, exponentially fast, as $n\rightarrow+\infty$, when $\alpha<2/3-\delta/2$
for any $\delta>0$.
\end{proof}

\section{Proof of Theorem \ref{thm:main} \label{Sec:mainproof}}

The beginning of the proof is similar to a corresponding proof in
\cite{HoudreLCSVARLB2012}. For a random variable $U$ with finite
$r$-th moment and for a random vector $V$, let $\mathbb{M}_{r}(U|V):=\mathbb{E}(|U-\mathbb{E}(U|V)^{r}|V$).
Clearly, by convexity and the conditional Jensen's inequality,
\begin{eqnarray}
\mathbb{M}_{r}(U|V) & \le & 2^{r}\left(\left(\mathbb{E}\left(|U-\mathbb{E}U|^{r}|V\right)\right)/2+\mathbb{E}\left(|\mathbb{E}\left(U|V\right)-\mathbb{E}U|^{r}|V\right)/2\right)\nonumber \\
 & \le & 2^{r}\mathbb{E}\left(|U-\mathbb{E}U|^{r}|V\right),\label{eq:condrmom}
\end{eqnarray}
and so, for any $n\ge1$,
\begin{eqnarray}
\mathbb{M}_{r}\left(L_{n}\left(N\right)\right) & \ge & \frac{1}{2^{r}}\mathbb{E}\left(\mathbb{M}_{r}\left(L_{n}\left(N\right)\right)|\left(L_{n}(k)\right)_{0\le k\le n^{1+\alpha}}\right)\nonumber \\
 & = & \frac{1}{2^{r}}\int_{\Omega_{n}}\mathbb{M}_{r}\left(L_{n}\left(N\right)|\left(L_{n}(k)\right)_{0\le k\le n^{1+\alpha}}(\omega)\right)\mathbb{P}(d\omega)\nonumber \\
 & \ge & \frac{1}{2^{r}}\int_{O_{n}}\mathbb{M}_{r}\left(L_{n}\left(N\right)|\left(L_{n}(k)\right)_{0\le k\le n^{1+\alpha}}(\omega)\right)\mathbb{P}(d\omega).\label{eq:condrmomlb}
\end{eqnarray}
Moreover, since $N$ is independent of $\left(L_{n}(k)\right)_{0\le k\le n^{1+\alpha}}$,
and from (\ref{eq:condrmom}), for each $\omega\in\Omega_{n}$,
\begin{align}
 & \mathbb{M}_{r}\left(L_{n}\left(N\right)|\left(L_{n}(k)\right)_{0\le k\le n^{1+\alpha}}(\omega)\right)\nonumber \\
\ge & \ \frac{1}{2^{r}}\mathbb{M}_{r}\left(L_{n}\left(N\right)|\left(L_{n}(k)\right)_{0\le k\le n^{1+\alpha}}(\omega),\boldsymbol{1}_{N\in I}=1\right)\mathbb{P}\left(N\in I|\left(L_{n}(k)\right)_{0\le k\le n^{1+\alpha}}(\omega)\right)\nonumber \\
= & \ \frac{1}{2^{r}}\mathbb{M}_{r}\left(L_{n}\left(N\right)|\left(L_{n}(k)\right)_{0\le k\le n^{1+\alpha}}(\omega),\boldsymbol{1}_{N\in I}=1\right)\mathbb{P}\left(N\in I\right).\label{eq:condrmomdecomp1}
\end{align}
In addition (see \cite[Lemma 2.2]{HoudreLCSVARLB2012}), note that
if $f:D\rightarrow\mathbb{Z}$ satisfies a local reversed Lipschitz
condition, i.e., $f$ is such that for any $i,j\in D$ with $j>i+\ell$,
$\ell\ge0$, $f(j)-f(i)\ge c(j-i)$ for some $c>0$ and if $T$ is
a $D$-valued random variable with $\mathbb{E}|f(T\}|^{r}<+\infty$,
$r\ge1$, then 
\[
\mathbb{M}_{r}(f(T))\ge\left(\frac{c}{2}\right)^{r}\left(\mathbb{M}_{r}(T)-\ell^{2}\right).
\]
So, for each $\omega\in O_{n}$, since $N$ is independent of $\left(L_{n}(k)\right)_{0\le k\le n^{1+\alpha}}$,
\begin{equation}
\mathbb{M}_{r}\left(L_{n}\left(N\right)|\left(L_{n}(k)\right)_{0\le k\le n^{1+\alpha}}(\omega),\boldsymbol{1}_{N\in I}=1\right)\ge\left(\frac{c_{3}}{n^{\alpha}}\right)^{r}\left(\mathbb{M}_{r}\left(N|\boldsymbol{1}_{N\in I}=1\right)-\ell(n)^{r}\right).\label{eq:condmomlbafterlip}
\end{equation}
Next, (\ref{eq:condrmomlb}), (\ref{eq:condrmomdecomp1}) and (\ref{eq:condmomlbafterlip})
lead to
\begin{equation}
\mathbb{M}_{r}\left(L_{n}(N)\right)\ge\frac{c_{3}^{r}}{2^{r}n^{r\alpha}}\left(\mathbb{M}_{r}\left(N|\boldsymbol{1}_{N\in I}=1\right)-\ell(n)^{r}\right)\mathbb{P}\left(N\in I\right)\mathbb{P}\left(O_{n}\right),\label{eq:condmomlbafterlip2}
\end{equation}
and it remains to estimate the first two terms on the right side of
(\ref{eq:condmomlbafterlip2}). By the Berry-Ess\'{e}en Theorem,
and for all $n\ge1$,
\begin{equation}
\left\lvert \mathbb{P}\left(N\in I\right)-\frac{1}{\sqrt{2\pi}}\int_{-1}^{1}\exp\left(-\frac{x^{2}}{2}\right)dx\right\rvert \le\frac{1}{\sqrt{n^{1+\alpha}p(1-p)}}.\label{eq:Berryesson}
\end{equation}
On the other hand,

\begin{align}
\mathbb{M}_{r}\left(N|\boldsymbol{1}_{N\in I}=1\right) & =\mathbb{E}\left(\left\lvert N-n^{1+\alpha}p+n^{1+\alpha}p-\mathbb{E}\left(N\vert\boldsymbol{1}_{N\in I}=1\right\rvert ^{r}\right)\vert\boldsymbol{1}_{N_{1}\in I}=1\right)\nonumber \\
 & \ge\left\lvert \mathbb{E}\left(\left\lvert N-n^{1+\alpha}p\right\rvert \vert\boldsymbol{1}_{N\in I}=1\right)^{1/r}-\left\lvert n^{1+\alpha}p-\mathbb{E}\left(N\vert\boldsymbol{1}_{N\in I}=1\right)\right\rvert \right\rvert ^{r},\label{eq:N1mom}
\end{align}
and when $n$ is large enough,
\begin{align}
 & \left\lvert n^{1+\alpha}p-\mathbb{E}\left(N\vert\boldsymbol{1}_{N\in I}=1\right)\right\rvert \nonumber \\
= & \sqrt{n^{1+\alpha}p(1-p)}\left\lvert \mathbb{E}\left(\frac{N-n^{1+\alpha}p}{\sqrt{n^{1+\alpha}p(1-p)}}\vert\boldsymbol{1}_{N\in I}=1\right)\right\rvert \nonumber \\
= & \sqrt{n^{1+\alpha}p(1-p)}\frac{\left\lvert F_{n}(1)-\Phi(1)+F_{n}(-1)-\Phi(-1)-\int_{-1}^{1}\left(F_{n}(x)-\Phi(x)\right)dx\right\rvert }{\mathbb{P}\left(N\in I\right)}\nonumber \\
\le & \sqrt{n^{1+\alpha}p(1-p)}\frac{4\max_{x\in[-1,1]}\left\lvert F_{n}(x)-\Phi(x)\right\rvert }{\mathbb{P}(N\in I)}\nonumber \\
\le & \sqrt{n^{1+\alpha}p(1-p)}\frac{2/\sqrt{n^{1+\alpha}p(1-p)}}{\int_{-1}^{1}\exp\left(-\frac{x^{2}}{2}\right)dx/\sqrt{2\pi}-1/\sqrt{n^{1+\alpha}p(1-p)}}\nonumber \\
\le & \frac{3}{\int_{-1}^{1}\exp\left(-\frac{x^{2}}{2}\right)dx/\sqrt{2\pi}},\label{eq:N1mom1}
\end{align}
where $F_{n}$ is the distribution function of $(N-n^{1+\alpha}p)/\sqrt{n^{1+\alpha}p(1-p)}$,
while $\Phi$ is the one of the standard Gaussian. Likewise,
\begin{align}
\mathbb{E} & \left(\left\lvert N-n^{1+\alpha}p\right\rvert ^{r}\vert\boldsymbol{1}_{N\in I}=1\right)\nonumber \\
 & \ge\left(n^{1+\alpha}p(1-p)\right)^{r/2}\frac{\int_{-1}^{1}\lvert x\rvert^{r}d\Phi(x)-4\max_{x\in[-1,1]}\left\lvert F_{n}(x)-\Phi(x)\right\rvert }{\mathbb{P}(N\in I)}\nonumber \\
 & \ge\left(n^{1+\alpha}p(1-p)\right)^{r/2}\frac{\int_{-1}^{1}\lvert x\rvert^{r}d\Phi(x)-2\sqrt{\pi}/\sqrt{n^{1+\alpha}p(1-p)}}{\int_{-1}^{1}\exp\left(-\frac{x^{2}}{2}\right)dx+\sqrt{\pi}/\sqrt{n^{1+\alpha}p(1-p)}}\nonumber \\
 & \ge\left(n^{1+\alpha}p(1-p)\right)^{r/2}\frac{\int_{-1}^{1}\lvert x\rvert^{r}d\Phi(x)}{2\int_{-1}^{1}\exp\left(-\frac{x^{2}}{2}\right)dx}.\label{eq:N1mom2}
\end{align}
Next, (\ref{eq:N1mom}), (\ref{eq:N1mom1}) and (\ref{eq:N1mom2})
give
\begin{align}
 & \mathbb{M}_{r}\left(N_{1}\vert\boldsymbol{1}_{N_{1}\in I}=1\right)\nonumber \\
\ge & n^{\frac{r(1+\alpha)}{2}}\left\lvert \sqrt{p(1-p)}\left(\frac{\int_{-1}^{1}\lvert x\rvert^{r}d\Phi(x)}{2\int_{-1}^{1}\exp\left(-\frac{x^{2}}{2}\right)dx}\right)^{1/r}-\frac{3}{\int_{-1}^{1}\exp\left(-\frac{x^{2}}{2}\right)dx/\sqrt{2\pi}}\right\rvert ^{r}.\label{eq:N1momfinal}
\end{align}
For $\mathbb{M}_{r}\left(N\vert\boldsymbol{1}_{N\in I}=1\right)$
to dominate the first term $\mathbb{M}_{r}\left(N|\boldsymbol{1}_{N\in I}=1\right)-\ell(n)^{r}$
in (\ref{eq:condmomlbafterlip2}), the constant $c_{1}$ (which depends
on $r$ and $p$ but not $n$) is chosen such that:
\[
c_{1}(r)\le\sqrt{p(1-p)}\left(\frac{\int_{-1}^{1}\lvert x\rvert^{r}d\Phi(x)}{2\int_{-1}^{1}\exp\left(-\frac{x^{2}}{2}\right)dx}\right)^{1/r}.
\]
(Recall that $\ell(n)=c_{1}n^{(1+\alpha)/2}$). So,
\[
\mathbb{M}_{r}\left(N|\boldsymbol{1}_{N\in I}=1\right)-\ell(n)^{r}\ge n^{r(1+\alpha)/2}\left(\sqrt{p(1-p)}\left(\frac{\int_{-1}^{1}\lvert x\rvert^{r}d\Phi(x)}{2\int_{-1}^{1}\exp\left(-\frac{x^{2}}{2}\right)dx}\right)^{1/r}-c_{1}\right)^{r}.
\]
This last estimate combined with (\ref{eq:Berryesson}) and Lemma
\ref{lem:lbOn} gives
\begin{eqnarray*}
\mathbb{M}_{r}\left(L_{n}(N)\right) & \ge & \frac{c_{3}^{r}}{2^{r}n^{r\alpha}}\left(\mathbb{M}_{r}\left(N|\boldsymbol{1}_{N\in I}=1\right)-\ell(n)^{r}\right)\mathbb{P}\left(N\in I\right)\mathbb{P}\left(O_{n}\right)\\
 & \ge & \frac{c_{3}^{r}}{2^{r}n^{r\alpha}}\left(\frac{1}{2\sqrt{2\pi}}\int_{-1}^{1}\exp\left(-\frac{x^{2}}{2}\right)dx\right)\\
 &  & n^{r(1+\alpha)/2}\left(\sqrt{p(1-p)}\left(\frac{\int_{-1}^{1}\lvert x\rvert^{r}d\Phi(x)}{2\int_{-1}^{1}\exp\left(-\frac{x^{2}}{2}\right)dx}\right)^{1/r}-c_{1}\right)^{r}\\
 &  & \left(1-12p(1-p)n^{1+\alpha}\exp(-c_{2}n)+p(1-p)n^{1+\alpha}\exp\left(-\frac{c_{5}^{2}c_{6}\sqrt{p(1-p)}}{2C^{2}}n^{\frac{1-3\alpha}{2}}\right)\right)\\
 & = & \Theta\left(n^{(1-\alpha)r/2}\right).
\end{eqnarray*}

\section{Conclusions and Remarks}

The major limitation of our method is the upper bound $1/3$ on $\alpha$,
which stems from application of Hoeffding's classical martingale exponential
inequality. Specifically, we note there is some discrepancy between
the orders of the upper and lower bounds for the martingale differences
in (\ref{eq:martingaleroot}) conditioned on the event $O_{n}$, i.e.,
the conditional lower bound is of order $o(n^{-\alpha})$ compared
to the upper bound $o(1)$. With the existence of this discrepancy,
it takes exactly $\alpha<1/3$ to have the exponential concentration
hold. But a more sophisticated way of flipping weights from $lo$
mode to $hi$ mode in the construction of the martingale might be
produced to mitigate this so as to relieve the $1/3$ bound. Or even
better, a more powerful concentration inequality can be used to replace
Hoeffding's. 

However, even if our method is generalizable to the case when $\alpha=1$,
i.e., the grid is perfect square, the corresponding lower bound for
the variance will be $O(n^{1-\alpha=1})=O(1)$ and thus not useful.
Nevertheless, a well-known fact that geodesics in $DLPP$ are confined
to a cylinder centered on the main diagonal of the grid and of width
of order strictly smaller than $o(n)$ will help producing a non-trivial
lower bound. The typical order of the width of the cylinder is the
transversal fluctuation, which is believed to be $n^{2/3}$. Further,
it is also believed that there is exponentially high probability that
geodesics are confined to such kind of cylinder of width $o(n^{2/3+\epsilon})$,
for $\epsilon>0$. Actually it has been proved that the transversal
fluctuation exponent can be upper bounded by $3/4$ in the setting
of undirected first passage percolation in \cite{newman1995divergence}
and an exponential concentration holds for all the geodesics in a
cylinder of width $O(n^{(2\kappa+2)/(2\kappa+3)}\sqrt{\ln n})$ in
\cite{hx2015} in the current setting, both of which assume the finite
curvature exponent $\kappa$. This is equivalently to say that if
let $\tilde{L}_{n}$ be the last passage time within the cylinder,
then $\tilde{L}_{n}\ge L_{n}$ holds with exponentially high probability.
So 
\begin{align*}
\mathbb{E}\tilde{L}_{n}-\mathbb{E}L_{n} & =\mathbb{E}\left((\tilde{L}_{n}-L_{n})\left(\boldsymbol{1}_{\{\tilde{L}_{n}\ge L_{n}\}}+\boldsymbol{1}_{\{\tilde{L}_{n}<L_{n}\}}\right)\right)\\
 & =\mathbb{E}\left((\tilde{L}_{n}-L_{n})\boldsymbol{1}_{\{\tilde{L}_{n}<L_{n}\}}\right)\ge-2n\mathbb{P}(\tilde{L}_{n}<L_{n})\rightarrow0.
\end{align*}
Meanwhile, it is trivial that $\tilde{L}_{n}\le L_{n}$. So $\mathbb{E}\tilde{L}_{n}-\mathbb{E}L_{n}\rightarrow0$
exponentially fast. This shows the potential of bounding the variance
of $L_{n}$ by that of $\tilde{L}_{n}$. Indeed,
\begin{align*}
Var\left(L_{n}\right) & =Var\left(L_{n}-\mathbb{E}\tilde{L}_{n}\right)\\
 & =\mathbb{E}\left(L_{n}-\mathbb{E}\tilde{L}_{n}\right)^{2}-\left(\mathbb{E}L_{n}-\mathbb{E}\tilde{L}_{n}\right)^{2}\\
 & =\text{\ensuremath{\mathbb{E}}}\left(\left(L_{n}-\mathbb{E}\tilde{L}_{n}\right)\left(\boldsymbol{1}_{\{\tilde{L}_{n}\ge L_{n}\}}+\boldsymbol{1}_{\{\tilde{L}_{n}<L_{n}\}}\right)\right)-\left(\mathbb{E}L_{n}-\mathbb{E}\tilde{L}_{n}\right)^{2}\\
 & =\mathbb{E}\left(\left(\tilde{L}_{n}-\mathbb{E}\tilde{L}_{n}\right)^{2}\boldsymbol{1}_{\{\tilde{L}_{n}\ge L_{n}\}}\right)+\mathbb{E}\left(\left(L_{n}-\mathbb{E}\tilde{L}_{n}\right)\boldsymbol{1}_{\{\tilde{L}_{n}<L_{n}\}}\right)-\left(\mathbb{E}L_{n}-\mathbb{E}\tilde{L}_{n}\right)^{2}\\
 & =Var(\tilde{L}_{n})+\mathbb{E}\left(\left(\left(L_{n}-\mathbb{E}\tilde{L}_{n}\right)^{2}-\left(\tilde{L}_{n}-\mathbb{E}\tilde{L}_{n}\right)^{2}\right)\boldsymbol{1}_{\{\tilde{L}_{n}<L_{n}\}}\right)-\left(\mathbb{E}L_{n}-\mathbb{E}\tilde{L}_{n}\right)^{2}\\
 & \ge Var(\tilde{L}_{n})-8n^{2}\mathbb{P}(\tilde{L}_{n}<L_{n})-\left(\mathbb{E}L_{n}-\mathbb{E}\tilde{L}_{n}\right)^{2}.
\end{align*}
Symmetrically, it is also true that $VarL_{n}\le Var\tilde{L}_{n}+8n^{2}\mathbb{P}(\tilde{L}_{n}<L_{n})-\left(\mathbb{E}L_{n}-\mathbb{E}\tilde{L}_{n}\right)^{2}.$
So the variances of $L_{n}$ and $\tilde{L}_{n}$ share the same asymptotic
order. On the other hand, our method here for the thin rectangle applies
to the cylinder of the length $O(n)$ and the width $O(n^{\alpha})$
with slight modification. This will produce a power lower bound $n^{1-\alpha}$.
Considering the best the scenario, if exponential concentration for
the width $n^{2/3+\epsilon}$ for any $\epsilon>0$ can be proved,
the corresponding power lower bound for longitudinal fluctuation will
be $n^{1-2/3-\epsilon}=n^{1/3-\epsilon}$. Although this is still
not the tight conjectured bound $n^{2/3}$, it still serves as a good
power lower bound.

\bibliographystyle{plain}
\phantomsection\addcontentsline{toc}{section}{\refname}\bibliography{percolation}

\begin{thebibliography}{10}

\bibitem{auffinger2013differentiability}
Antonio Auffinger and Michael Damron.
\newblock Differentiability at the edge of the percolation cone and related
  results in first-passage percolation.
\newblock {\em Probability Theory and Related Fields}, 156(1-2):193--227, 2013.

\bibitem{auffinger201550}
Antonio Auffinger, Michael Damron, and Jack Hanson.
\newblock {\em 50 Years of First-Passage Percolation}, volume~68 of {\em
  University Lecture Series}.
\newblock American Mathematical Society, 2017.

\bibitem{baik2001optimal}
Jinho Baik, Percy Deift, Ken McLaughlin, Peter Miller, and Xin Zhou.
\newblock Optimal tail estimates for directed last passage site percolation
  with geometric random variables.
\newblock {\em Advances in Theoretical and Mathematical Physics}, 5(6):1--41,
  2001.

\bibitem{benaim2008exponential}
Michel Benaim and Rapha{\"e}l Rossignol.
\newblock Exponential concentration for first passage percolation through
  modified poincar{\'e} inequalities.
\newblock In {\em Annales de l'IHP Probabilit{\'e}s et statistiques},
  volume~44, pages 544--573, 2008.

\bibitem{benjamini2011first}
Itai Benjamini, Gil Kalai, and Oded Schramm.
\newblock First passage percolation has sublinear distance variance.
\newblock {\em The Annals of Probability}, 31(4):1970--1978, 2003.

\bibitem{bodineau2005universality}
Thierry Bodineau and James Martin.
\newblock A universality property for last-passage percolation paths close to
  the axis.
\newblock {\em Electron. Comm. Probab}, 10:105--112, 2005.

\bibitem{boucheron2013concentration}
St{\'e}phane Boucheron, G{\'a}bor Lugosi, and Pascal Massart.
\newblock {\em Concentration inequalities: A nonasymptotic theory of
  independence}.
\newblock OUP Oxford, 2013.

\bibitem{chatterjee2017general}
Sourav Chatterjee.
\newblock A general method for lower bounds on fluctuations of random
  variables.
\newblock {\em arXiv preprint arXiv:1706.04290}, 2017.

\bibitem{damron2014subdiffusive}
Michael Damron, Jack Hanson, and Philippe Sosoe.
\newblock Subdiffusive concentration in first-passage percolation.
\newblock {\em Electron. J. Probab}, 19(109):1--27, 2014.

\bibitem{HoudreCLT2014}
Christian {Houdr{\'e}} and {\"U}mit {I{\c s}lak}.
\newblock {A central limit theorem for the length of the longest common
  subsequences in random words}.
\newblock {\em arXiv preprint arXiv: 1408.1559v4}, 2017.

\bibitem{HoudreLCSVARLB2012}
Christian {Houdr{\'e}} and Jingyong {Ma}.
\newblock {On the order of the central moments of the length of the longest
  common subsequences in random words}.
\newblock {\em High Dimensinal Probability: The Carg{\`e}se Volume}, Progress
  in Probability 71, Birkh{\"a}user:105--137, 2016.

\bibitem{hx2015}
Christian Houdr{\'e} and Chen Xu.
\newblock Concentration of geodesics in directed bernoulli percolation.
\newblock {\em arXiv preprint arXiv:1607.02219}, 2016.

\bibitem{johansson2000shape}
Kurt Johansson.
\newblock Shape fluctuations and random matrices.
\newblock {\em Communications in Mathematical Physics}, 209(2):437--476, 2000.

\bibitem{johansson2000transversal}
Kurt Johansson.
\newblock Transversal fluctuations for increasing subsequences on the plane.
\newblock {\em Probability Theory and Related Fields}, 116(4):445--456, 2000.

\bibitem{kesten1986aspects}
Harry Kesten.
\newblock Aspects of first passage percolation.
\newblock In {\em {\'E}cole d'{\'E}t{\'e} de Probabilit{\'e}s de Saint Flour
  XIV-1984}, pages 125--264. Springer, 1986.

\bibitem{kesten1993speed}
Harry Kesten.
\newblock On the speed of convergence in first-passage percolation.
\newblock {\em The Annals of Applied Probability}, 3(2):296--338, 1993.

\bibitem{lember2009standard}
J{\"u}ri Lember and Heinrich Matzinger.
\newblock Standard deviation of the longest common subsequence.
\newblock {\em The Annals of Probability}, 37(3):1192--1235, 2009.

\bibitem{martin2004limiting}
James~B. Martin.
\newblock Limiting shape for directed percolation models.
\newblock {\em The Annals of Probability}, 32(4):2908--2937, 2004.

\bibitem{newman1995divergence}
Charles~M. Newman and Marcelo~S.T. Piza.
\newblock Divergence of shape fluctuations in two dimensions.
\newblock {\em The Annals of Probability}, 23(3):977--1005, 1995.

\bibitem{pemantle1994planar}
Robin Pemantle and Yuval Peres.
\newblock Planar first-passage percolation times are not tight.
\newblock In {\em Probability and phase transition}, pages 261--264. Springer,
  1994.

\bibitem{zhang2008shape}
Yu~Zhang.
\newblock Shape fluctuations are different in different directions.
\newblock {\em The Annals of Probability}, 36(1):331--362, 2008.

\end{thebibliography}

\end{document}